\documentclass[11pt]{amsart}
\usepackage[english]{babel}
\usepackage[a4paper,top=2cm,bottom=2cm,left=3cm,right=3cm,marginparwidth=1.75cm]{geometry}
\usepackage{amsmath}
\usepackage{amssymb}
\usepackage{amsthm}
\usepackage{mathtools}
\usepackage{graphicx}
\usepackage[colorlinks=true, allcolors=blue]{hyperref}
\usepackage{tikz-cd}
\usepackage{enumerate}

\newtheorem{theorem}{Theorem}[section] 
\newtheorem{corollary}[theorem]{Corollary} 
\newtheorem{lemma}[theorem]{Lemma} 

\newtheorem{ques}{Question}[section]

\begin{document}

\title[On $C^k$-functions mapping $\mathbb{Q}$ into itself and Mahler's problem]{On $C^k$-functions mapping $\mathbb{Q}$ into itself and Mahler's problem on Liouville numbers}

\author[J. Lelis]{Jean Lelis}
\address{Departamento de Matemática, Universidade de Brasília, Brasília - DF, Brazil.}
\email{jean.lelis@unb.br}

\author[C. G. Moreira]{Carlos Gustavo Moreira}
\address{Instituto de Matemática Pura e Aplicada IMPA, Rio de Janeiro-RJ, Brasil.}
\email{gugu@impa.br}

\author[E. Silva]{Elaine Silva}
\address{Instituto de Matemática, Universidade Federal de Alagoas, Maceió - AL, Brazil}
\email{elaine.silva@im.ufal.br}

\subjclass[2020]{11B37, 11B39}

\keywords{}
    
\begin{abstract}
Liouville numbers form a classical class of transcendental real numbers characterized by exceptionally strong rational approximations. A theorem of Maillet shows that non-constant rational functions with rational coefficients preserve the Liouville property, motivating a question of Mahler on whether analogous phenomena hold for transcendental functions. In this paper, we address this problem for real functions of finite smoothness. For any $\varepsilon>0$, we construct an uncountable set of $C^k$-functions on $\mathbb{R}$, dense with respect to the topology of uniform convergence on compact sets, mapping $\mathbb{Q}$ into itself and satisfying $\operatorname{den}(f(p/q)) \le q^{2k+\varepsilon}$, and deduce that such functions preserve Liouville numbers. In contrast, we prove a rigidity result about a $C^{2k+1}$-function mapping $\mathbb{Q}$ into itself and satisfying $\operatorname{den}(f(p/q)) \ll q^k$.
\end{abstract}

\maketitle

\section{Introduction}

The arithmetic study of sufficiently smooth functions $f(x)$ mapping algebraic or rational numbers into themselves has been considered in several works. In 1886, Weierstrass gave an example (published in \cite{Wei23}) of a transcendental entire function $f$ that maps the set of algebraic numbers into itself and, moreover, satisfies $f(\mathbb{Q})\subseteq \mathbb{Q}$. Various other constructions followed. For instance, van der Poorten \cite{vdP68} showed that there exist transcendental analytic functions whose values, together with all their derivatives, map every number field into itself. More recently, Marques and Moreira \cite{MM17} gave a positive answer to a question of Mahler \cite[p.~53]{Mahler76} by proving the existence of uncountably many transcendental entire functions with rational coefficients such that both the image and the preimage of algebraic numbers under $f$ consist of algebraic numbers.

In 1991, Pila \cite{Pila91} developed a method to bound the number of rational (and, more generally, algebraic) points on the graph of a smooth function, refining earlier work with Bombieri \cite{bombieri1989number}. His approach is based on the study of how the graph of a function intersects algebraic curves of bounded degree.

In the case where $f$ is transcendental and analytic on a compact interval, one further uses that its graph is not contained in any algebraic curve and intersects each such curve only finitely many times. As a consequence, for every $\varepsilon>0$, there exists $c=c(f,\varepsilon)$ such that the number of rational points on the graph of $f$ of height at most $N$ satisfies
\[
\#\{(x,f(x))\in\mathbb{Q}^2 : H(x),H(f(x))\le N\} \le c N^{\varepsilon},
\]
where $H(p/q)=\max\{|p|,q\}$ for $p/q$ in lowest terms.

In 2016, Lombardo \cite{Lom17} studied analytic functions $f\colon[0,1]\to[0,1]$ inducing bijections $\mathbb{Q}\cap[0,1]\to\mathbb{Q}\cap[0,1]$, showing that such functions satisfy strong arithmetic constraints. For further developments, see \cite{MRS16,LM20}.

Functions mapping rational numbers into themselves exhibit remarkable arithmetic properties. In many cases, such functions preserve important number-theoretic structures, and their behavior is closely related to classical problems in Diophantine approximation and transcendence theory. Understanding these properties provides valuable insights into the interaction between analytic behavior and arithmetic constraints. In 2015, Marques and Moreira \cite{MM2015} related such functions to Mahler’s problem on \textit{Liouville numbers}.

In 1853, Liouville proved a fundamental theorem concerning approximations of algebraic numbers by rational numbers. This theorem enabled him to construct explicitly some transcendental numbers called Liouville numbers. A Liouville number is a real number that can be approximated exceptionally well by rationals. Specifically, $\xi$ is a Liouville number if there exists an infinite sequence of distinct rationals $p_n/q_n$ (with $q_n\to \infty$ as $n\to \infty$) such that 
\[
\left|\xi-\frac{p_n}{q_n}\right|<\frac{1}{q_n^{n}}
\] 
for every positive integer $n$. The set of all Liouville numbers is denoted by $\mathbb{L}$.

By Liouville's theorem, we know that all Liouville numbers are transcendental. In 1955, K. F. Roth \cite{roth55} proved that algebraic numbers cannot be ``well approximated'' by rationals. He proved that, given an irrational algebraic number $\alpha$ and any $\varepsilon>0$, the inequality  
\[
\left|\alpha-\frac{p}{q}\right|<\frac{1}{q^{2+\varepsilon}}
\]
has only finitely many solutions in non-zero $p,q\in\mathbb{Z}$. These results provide a natural way to construct examples of transcendental numbers using rational approximations.

In 1906, E. Maillet \cite{Maillet06} proved an important property of Liouville numbers: if $f(z)$ is any non-constant rational function with rational coefficients, then for every Liouville number $\xi$, the value $f(\xi)$ is itself a Liouville number. In other words, non-constant rational functions preserve the Liouville property. Maillet's idea was to use the fact that rational functions with rational coefficients map $\mathbb{Q}$ into itself in such a way that the \emph{height} of the image is polynomially bounded in terms of the height of the preimage, together with the mean value theorem. Here, the height of a rational number $x=p/q$, written in lowest terms with $p,q\in\mathbb{Z}$ and $q>0$, is defined by $H(x):=\max\{|p|,q\}$.

Motivated by Maillet's result, K. Mahler \cite{Mahler84} asked in 1984 whether {\it transcendental entire functions} could have a similar property. Mahler's question was:
\begin{ques}\label{qm}
Which analytic functions $f(z)$ have the property that if $\xi$ is any Liouville number, then so is $f(\xi)$? In particular, are there entire transcendental functions with this property?
\end{ques}

This became known as Mahler's problem on Liouville numbers. It essentially asks if a transcendental analytic function can carry $\mathbb{L}$ into itself. 

In 2015, Marques and Moreira \cite{MM2015} studied Mahler's question on Liouville numbers. They constructed uncountably many transcendental entire functions $f$ such that $f(\mathbb{Q})\subseteq \mathbb{Q}$ and $\operatorname{den}(f(p/q))<q^{8q^2}$ for every rational number $p/q$ (here, and in what follows, $\operatorname{den}(z)$ denotes the denominator of the rational number $z$ in lowest terms). Using these functions and following the ideas from Maillet's result, they proved the existence of a transcendental entire function $f$ such that $f(\mathbb{L}_{\text{ultra}})\subseteq\mathbb{L}_{\text{ultra}}$, where $\mathbb{L}_{\text{ultra}}\subseteq\mathbb{L}$ is a dense $G_{\delta}$ set (which means that $\mathbb{L}_{\text{ultra}}$ is a large set in a topological sense) formed by numbers with even better rational approximations than Liouville numbers.

In this paper, we establish some results on $C^k$-functions mapping rational numbers into themselves. First, we prove a rigidity result showing that a $C^k$-function mapping rational numbers into rational numbers with polynomially bounded heights and having rational derivatives at the origin up to order $k$ must be locally a rational function. More precisely, we prove the following result.

\begin{theorem}\label{theo1}
Let $\Omega$ be a connected neighborhood of the origin, let $k$ be a positive integer, and let
$f : \Omega \to \mathbb{R}$ be a function of class $C^{2k+1}$ on $\Omega$ such that
$f^{(j)}(0) \in \mathbb{Q}$ for all $0 \le j \le 2k$, 
$f(\Omega \cap \mathbb{Q}) \subseteq \mathbb{Q}$, and
\[
\operatorname{den}\left(f\left(\frac{p}{q}\right)\right)\ll q^k
\]
for all rational numbers $p/q \in \Omega$ with $q$ sufficiently large. Then $f$ is a rational function of rational coefficients with degree at most $k$.
\end{theorem}

The proof of the above theorem is inspired by the argument in \cite{LelisMarquesMoreiraTrojovsky2024}, where a similar result is established for analytic functions by constructing a rational function that coincides with $f$ on a set accumulating at zero. Since the identity theorem is not available for functions of class $C^{2k+1}$, we instead use estimates on divided differences to show that these functions actually coincide on a dense set. It then follows that they must be identical.

Furthermore, we endow the space $C^k(\mathbb{R})$ with its natural Fréchet topology, namely the topology of uniform convergence on compact sets of all derivatives up to order $k$. A convenient way to describe this topology is through the family of seminorms
\[
p_m(f):=\sum_{j=0}^{k}\sup_{|x|\le m}|f^{(j)}(x)| \qquad (m\in\mathbb{N}),
\]
which induces the metric
\[
d(f,g):=\sum_{m=1}^{\infty}2^{-m}\min\{1,p_m(f-g)\},
\qquad f,g\in C^k(\mathbb{R}).
\]
This metric is compatible with the usual topology of convergence in $C^k$ on compact subsets of $\mathbb{R}$. In other words, a sequence $(f_n)$ converges to $f$ with respect to $d$ if and only if, for every integer $0\le j\le k$ and every compact interval $[-m,m]$, the derivatives $f_n^{(j)}$ converge uniformly to $f^{(j)}$ on $[-m,m]$.

Thus, given $t>2k$, we construct uncountably many functions $f\in C^k(\mathbb{R})$ mapping $\mathbb{Q}$ into itself such that the denominator of $f(p/q)$ is bounded by $q^{t}$ for all rational numbers $p/q$ in lowest terms. In fact, we prove that the set of such functions is dense in $C^k(\mathbb{R})$ with respect to the topology of uniform convergence on compact sets. 

Moreover, we show that these functions can be constructed to coincide on the set
\[
C(2):=\left\{ x\in\mathbb{R} : x=[a_0;a_1,a_2,\ldots]\ \text{with}\ a_n\in\{1,2\}\ \text{for all } n\geq 1 \right\},
\]
that is, the set of real numbers whose continued fraction expansions have partial quotients bounded by $2$. In particular, at most one of these functions can be algebraic. More precisely, we prove the following theorem.

\begin{theorem}\label{theo2}
Let $k\ge 1$ be an integer and let $t>2k$ be a real number. Consider $C^k(\mathbb{R})$ endowed with the metric $d$ introduced above. Then, for every $g\in C^k(\mathbb{R})$ and every $\varepsilon>0$, there are uncountably many functions $f\in C^k(\mathbb{R})$ and a constant $C>0$ such that $d(f,g)<\varepsilon$, $f(\mathbb{Q})\subseteq \mathbb{Q}$, and
\[
\operatorname{den}\left(f\left(\frac{p}{q}\right)\right)\le C q^t
\]
for every rational number $p/q$ written in lowest terms. Moreover, $f(x)=g(x)$ for all $x\in C(2)$.
\end{theorem}

As an application of this result, we derive consequences for $C^k$ functions mapping $\mathbb{L}$ into itself. In particular, we provide an affirmative answer to Mahler’s problem on Liouville numbers within the class $C^k$. More precisely, we obtain the following corollary, which we shall prove in Section \ref{sec4}.

\begin{corollary}\label{corol1}
Let $k\geq1$ be an integer and let $\mathcal{L}_k\subseteq C^k(\mathbb{R})$ be the set 
\[
\mathcal{L}_k :=
\left\{ f\in C^k(\mathbb{R}) : f(\mathbb{L}) \subseteq \mathbb{L},\, f\; \text{transcendental}\right\}.
\]
Then $\mathcal{L}_k$ is uncountable and dense in $C^k(\mathbb{R})$ with respect to the topology of uniform convergence on compact sets.
\end{corollary}

\section{On the Rigidity of $C^{k}$ Functions Mapping $\mathbb{Q}$ into Itself}

In this section, we prove Theorem \ref{theo1}, which establishes a rigidity criterion for functions of class $C^{2k+1}$ mapping rational numbers into rational numbers. Before proceeding to the main proof, we provide a brief review of Newton interpolation and divided differences, as these classical tools will play a fundamental role in our arguments.

\subsection{Newton Interpolation and Divided Differences}

In order to prove Theorem~\ref{theo1}, we first review some basic facts about divided differences. Let $(\alpha_i)_{i \geq 0}$ be a sequence of distinct points in $\mathbb{R}$. Define the Newton basis polynomials by
\begin{equation*}
    \omega_{0,\alpha}(x) = 1 \quad \text{and} \quad \omega_{n,\alpha}(x) = \prod_{i=0}^{n-1}(x - \alpha_i)\quad \text{for all } n\geq 1.
\end{equation*}
Let $N_n(x)$ denote the Newton interpolating polynomial of degree at most $n$ associated with a function $f$, interpolating $f$ at the points $\alpha_0, \alpha_1, \ldots, \alpha_n$. This polynomial is given by
\begin{equation*}
    N_n(x) = \sum_{i=0}^n \Delta_f(\alpha_0, \ldots, \alpha_i)\, \omega_{i,\alpha}(x),
\end{equation*}
where $\Delta_f(\alpha_0, \ldots, \alpha_i)$ denotes the $i$th-order divided difference of $f$, defined recursively by
\[
\Delta_f(\alpha_0) = f(\alpha_0),
\]
and
\[
\Delta_f(\alpha_0, \ldots, \alpha_k) = \frac{\Delta_f(\alpha_1, \ldots, \alpha_k) - \Delta_f(\alpha_0, \ldots, \alpha_{k-1})}{\alpha_k - \alpha_0}, \quad \text{for all } k\geq 1.
\]

Divided differences, as the coefficients of the Newton interpolating polynomial, play a central role in numerical analysis, particularly in polynomial interpolation, approximation theory, and spline constructions; see \cite{deBoor05} for a survey. They also arise in various applications in combinatorics \cite{CF11,FL03}, number theory \cite{Xu2012}, and dynamical systems, as in \cite{MS25}.

An explicit expression for the divided difference of order $n$ is given by
\begin{equation}
    \Delta_f(\alpha_0, \ldots, \alpha_n) = \sum_{k=0}^n \frac{f(\alpha_k)}{\displaystyle\prod_{\substack{0 \leq j \leq n \\ j \neq k}} (\alpha_k - \alpha_j)}.
\end{equation}

The following lemma relates the highest-order divided difference to the $k$th derivative of $f$ under suitable regularity assumptions. Since this result is classical in the theory of divided differences, we omit the proof.

\begin{lemma}\label{lemmdd}
Let $f \in C^k([a, b])$, and let $\alpha_0, \alpha_1, \ldots, \alpha_k$ be distinct points in $[a, b]$. Then there exists a point $\xi \in (a, b)$ such that
\[
\Delta_f(\alpha_0, \alpha_1, \ldots, \alpha_k) = \frac{f^{(k)}(\xi)}{k!}.
\]
\end{lemma}

\subsection{Proof of Theorem \ref{theo1}}

Let $f : \Omega \to \mathbb{R}$ be a function of class $C^{2k+1}$ on $\Omega$ such that $f^{(j)}(0) \in \mathbb{Q}$ for all $0 \le j \le 2k$, $f(\Omega \cap \mathbb{Q}) \subseteq \mathbb{Q}$, and
\[
\operatorname{den}\left(f\left(\frac{p}{q}\right)\right)\ll q^k
\]
for all rational numbers $p/q \in \Omega$ with $q$ sufficiently large. Without loss of generality, we may assume that $f(0)=0$.

By Taylor's theorem applied to $f$ at the origin with the Lagrange form of the remainder, for each $x\in \Omega$, there exists $\xi_x$ between $0$ and $x$ such that 
\begin{equation}\label{fx}
    f(x)=\sum_{j=1}^{2k}\frac{f^{(j)}(0)}{j!}x^j+\frac{f^{(2k+1)}(\xi_x)}{(2k+1)!} x^{2k+1}
    =F_{2k}(x)+x^{2k+1}T(x),
\end{equation}
where $F_{2k}(x)=a_0+a_1x+\cdots+a_{2k}x^{2k}$ is a polynomial with rational coefficients, and $T(x)$ is a continuous function on $\Omega$.

Let us show, as in \cite{LelisMarquesMoreiraTrojovsky2024}, that there exist polynomials $P_k(x)$ and $Q_k(x)$ with integer coefficients and degrees at most $k$ such that
\[
R_{k}(x):=\frac{P_k(x)}{Q_k(x)}=F_{2k}(x)+x^{2k+1}K(x),
\]
where $x^{2k+1}K(x)$ is an analytic function in a neighborhood $\tilde\Omega\subseteq\Omega$ of the origin. 

Indeed, write $Q_k(x)=q_0+q_1x+\cdots+q_kx^k$, and consider the product 
\[
Q_k(x)F_{2k}(x)=b_1x+\cdots+b_{3k}x^{3k}.
\]
We seek integers $q_0,q_1,\ldots,q_k$, not all zero, such that $b_i\in\mathbb{Z}$ for all $1\leq i\leq 3k$, and $b_j=0$ for all $j\in [k+1, 2k]$. In other words, we look for a non-trivial integer solution to the $k\times(k+1)$ homogeneous linear system
\[
    \left\{\begin{array}{ccccccccc}
        a_{2k}q_0 & + & a_{2k-1}q_1 & + & \cdots & + & a_{k}q_k & = & 0 \\
        a_{2k-1}q_0 & + & a_{2k-2}q_1 & + & \cdots & + & a_{k-1}q_k & = & 0 \\
        \vdots & \vdots & \vdots & \vdots & \vdots & \ddots & \vdots & \vdots \\
        a_{k+1}q_0 & + & a_{k}q_1 & + & \cdots & + & a_1q_k & = & 0
    \end{array}\right.
\]
in the variables $q_0,q_1,\ldots,q_k\in\mathbb{Z}$. Since the number of variables exceeds the number of equations, a basic result from linear algebra ensures the existence of a non-trivial solution in $\mathbb{Q}^{k+1}$. Clearing denominators, we may assume that $(q_0,\ldots,q_k)\in \mathbb{Z}^{k+1}\setminus \{(0,\ldots, 0)\}$. 

Therefore,
\[
Q_k(x)F_{2k}(x)=b_1x+b_2x^2+\cdots+b_kx^k+x^{2k+1}S(x),
\]
where $S(x)$ is a polynomial with rational coefficients of degree at most $k-1$. Hence,
\begin{equation}\label{eqRt}
   R_k(x)=\frac{P_k(x)}{Q_k(x)}=F_{2k}(x)-x^{2k+1}\frac{S(x)}{Q_k(x)}
\end{equation}
for all $x\in\Omega\setminus \{x \in \mathbb{R} : Q_k(x) = 0\}$, where $P_k(x)=b_1x+\cdots+b_kx^k$. 
Since $Q_k(x)$ is a non-zero polynomial of degree at most $k$, the multiplicity of $0$ as a root of $Q_k(x)$ is at most $k$. Thus, the highest power of $x$ dividing $Q_k(x)$ is at most $k$, whereas the numerator of the second term contains the factor $x^{2k+1}$. This guarantees that the rational function $x^{2k+1}S(x)/Q_k(x)$ has a removable singularity at the origin (in fact, a zero of multiplicity at least $k+1$). Consequently, $R_k(x)$ can be analytically continued to a neighborhood of the origin.

Now, we consider the function $g:\tilde\Omega\to \mathbb{R}$ defined by
\[
g(x):=f(x)-\tilde R_k(x)=x^{2k+1}(T(x)+K(x)),
\]
where $\tilde{\Omega} = \Omega \setminus \{x \in \mathbb{R} : x \neq 0, \, Q_k(x) = 0\}$, $K(x)=S(x)/Q_k(x)$ and $\tilde R_k(x)$ is the analytic continuation of $R_k(x)$ to a neighborhood of the origin. It follows from the definition of $g$ that it satisfies the following properties:
\begin{enumerate}[(i)]
    \item $g(\mathbb{Q}\cap\tilde\Omega)\subseteq \mathbb{Q}$;
    \item $\operatorname{den} \bigl(g(p/q)\bigr)\ll q^{2k}$;
    \item $g^{(m)}(0)=0$ for all $0\leq m \leq 2k$;
    \item $g\in C^{2k+1}(\tilde\Omega)$;
    \item There exist $\varepsilon>0$ and a constant $C=C(\varepsilon)>0$ such that 
    \[
    |g^{(2k+1)}(x)|\leq C \quad \text{for all } x\in[-\varepsilon, \varepsilon]\subset\tilde\Omega.
    \]
\end{enumerate}

First, we shall prove that, for each positive integer $n$, there exists a positive integer $M_n$ such that
\[
g\left(\frac{n}{q}\right)=0,\quad \text{for all}\quad q>M_n.
\]
    
Indeed, by \eqref{fx} and \eqref{eqRt}, for a fixed positive integer $n$, we have
\[
\left|g\left(\frac{n}{q}\right)\right|
=
\left|\left(\frac{n}{q}\right)^{2k+1}\left(T\left(\frac{n}{q}\right)+K\left(\frac{n}{q}\right)\right)\right|
\]
for all positive integers $q$ such that $n/q\in\tilde\Omega$. Therefore, if $g(n/q)\neq 0$, then
\[
\operatorname{den}\left(g\left(\frac{n}{q}\right)\right)
\left|\left(\frac{n}{q}\right)^{2k+1}\left(T\left(\frac{n}{q}\right)+K\left(\frac{n}{q}\right)\right)\right|
\geq 1.
\]
Since $\operatorname{den}(g(n/q))\ll q^{2k}$ and $(T+K)$ is bounded near the origin, the above inequality cannot hold for all sufficiently large $q$. Hence, there exists a positive integer $M_n$ such that $g(n/q)=0$ for all $q> M_n$.

Now, we want to prove that $g|_{\Omega'}\equiv0$ for some neighborhood $\Omega'\subseteq\Omega$ of the origin. Indeed,  let $\alpha_0,\ldots,\alpha_{2k+1}$ be $2k+1$ distinct points in $[-\varepsilon,\varepsilon]$. Consider the divided difference of order $2k+1$ given by
\begin{equation}\label{Dg1}
    \Delta_g(\alpha_0, \ldots, \alpha_{2k+1}) 
    = \sum_{i=0}^{2k+1} \frac{g(\alpha_i)}{\displaystyle\prod_{\substack{0 \leq j \leq 2k+1 \\ j \neq i}} (\alpha_i - \alpha_j)}.
\end{equation}

Suppose that $g|_{[-\varepsilon,\varepsilon]}\not\equiv0$. Fix a positive integer $n$, and let $M_n$ be the largest integer such that $n/M_n\in(-\varepsilon,\varepsilon)$ and $g(n/M_n)\neq 0$. Taking $\alpha_i=\frac{n}{M_n+i}$ in \eqref{Dg1}, we have $g(\alpha_i)=0$ for all $i>0$, and thus
\begin{equation}\label{Dg3}
    \Delta_g\left(\frac{n}{M_n}, \frac{n}{M_n+1},\ldots, \frac{n}{M_n+2k+1}\right) 
    = \frac{g\left(\frac{n}{M_n}\right)}{\displaystyle\prod_{j=1}^{2k+1} \left(\frac{n}{M_n}-\frac{n}{M_n+j}\right)}.
\end{equation}

A direct computation shows that
\[
\prod_{j=1}^{2k+1} \left(\frac{n}{M_n}-\frac{n}{M_n+j}\right)
=
\frac{n^{2k+1}}{M_n^{2k+1}\binom{M_n+2k+1}{M_n}}.
\]
On the other hand, by Lemma~\ref{lemmdd}, we have
\begin{equation}\label{Dg2}
\Delta_g\left(\frac{n}{M_n}, \frac{n}{M_n+1},\ldots, \frac{n}{M_n+2k+1}\right) 
= \frac{g^{(2k+1)}(\xi_n)}{(2k+1)!},
\end{equation}
for some $\xi_n\in (-\varepsilon,\varepsilon)$. Combining \eqref{Dg3} and \eqref{Dg2}, we obtain
\[
\frac{g^{(2k+1)}(\xi_n)}{(2k+1)!}
=
\frac{g\left(\frac{n}{M_n}\right)\cdot M_n^{2k+1}\cdot \binom{M_n+2k+1}{M_n}}{n^{2k+1}}.
\]

Therefore,
\[
|g^{(2k+1)}(\xi_n)|
=
\frac{\left|g\left(\frac{n}{M_n}\right)\right|\cdot M_n^{2k+1}\cdot \binom{M_n+2k+1}{M_n}\cdot (2k+1)!}{n^{2k+1}}.
\]
Since $g(n/M_n)\neq 0$ and $\operatorname{den}(g(n/M_n))=O(M_n^{2k})$, we obtain
\[
\left|g\left(\frac{n}{M_n}\right)\right|\gg \frac{1}{M_n^{2k}}.
\]
Moreover, noting that $\binom{M_n+2k+1}{M_n} \ge 1$ and $M_n\gg n$, it follows that
\[
|g^{(2k+1)}(\xi_n)|
\gg
\frac{M_n^{2k+1}}{M_n^{2k}}
=
M_n.
\]
In particular, since $M_n \to \infty$ as $n \to \infty$, we have
\[
|g^{(2k+1)}(\xi_n)| \to \infty \quad \text{as } n\to\infty,
\]
which contradicts property (v). Therefore, we conclude that there exists an open interval $\Omega' \subseteq (-\varepsilon, \varepsilon)$ such that $f(x) = R_k(x)$ for all $x \in \Omega'$. 

To complete the proof, we must show that this equality holds on the entire domain $\Omega$. Let $U$ be the connected component of $\Omega \setminus \{x \in \mathbb{R} : x \neq 0, \, Q_k(x) = 0\}$ containing the origin. Since $R_k(x)$ is a rational function with a removable singularity at the origin, it is well-defined and analytic on $U$. 

The previous divided difference argument relied on the bound $\operatorname{den}(g(p/q)) \ll q^{2k}$ and the boundedness of $T(x) + K(x)$. Since these bounds hold uniformly on any compact subset of $U$, we can apply the exact same reasoning to deduce that $f(x) = R_k(x)$ on any compact interval $[a, b] \subset U$ containing the origin. By exhausting the open set $U$ with an increasing sequence of such compact intervals, it follows that $f(x) = R_k(x)$ for all $x \in U$.

Suppose, for the sake of contradiction, that $U \neq \Omega$. Since $\Omega$ is a connected neighborhood (an interval), $U$ must have at least one boundary point $t \in \Omega$. By the definition of $U$, this boundary point $t$ must be a root of $Q_k(x)$, meaning $t$ is a pole of $R_k(x)$. This implies that
\[
\lim_{\substack{s \to t \\ s \in U}} |R_k(s)| = +\infty.
\]

However, since $f$ is of class $C^{2k+1}$ on $\Omega$, it is continuous at $t$. Knowing that $f(x)$ coincides with $R_k(x)$ on $U$, we must have
\[
|f(t)| = \lim_{\substack{s \to t \\ s \in U}} |f(s)| = \lim_{\substack{s \to t \\ s \in U}} |R_k(s)| = +\infty,
\]
which is a contradiction, since $f(t)$ must be a finite real number. 

Therefore, our assumption that $U \neq \Omega$ is false. We conclude that $U = \Omega$, meaning $R_k(x)$ has no poles in $\Omega$, and $f(x) = R_k(x)$ for all $x \in \Omega$. This completes the proof.
\qed

\section{$C^k$-functions mapping $\mathbb{Q}$ into itself}\label{sec4}

In this section, we establish the main constructive arguments of this paper. In Section~\ref{sub3.1}, we prove Theorem \ref{theo2} by explicitly constructing a function $f$ through a properly scaled series of smooth bump functions. Recall that bump functions (smooth real-valued functions with compact support) are classical and ubiquitous tools in real analysis and differential topology (for an explicit construction and fundamental properties on $\mathbb{R}$, we refer the reader to \cite[Chapter 2]{Lee2012}). In Section~\ref{sub3.2}, we prove Lemma \ref{lemma:liouville-preservation}, which provides a general criterion ensuring that $C^1$-functions mapping rational numbers into rational numbers preserve Liouville-type properties under mild growth conditions on the denominators. Finally, in Section~\ref{sub3.3}, we combine these results to prove Corollary \ref{corol1}, demonstrating the density of transcendental functions within our constructed family.

\subsection{Proof of Theorem \ref{theo2}}\label{sub3.1}

Fix a bump function $\psi \in C^{\infty}(\mathbb{R})$ and a constant $M>0$ such that 
\begin{itemize}
    \item[(i)] $\operatorname{supp} \psi = \left[-\tfrac{1}{4}, \tfrac{1}{4}\right]$;
    \item[(ii)] $\psi(x) = 1$ for all $x \in \left(-\tfrac{1}{5}, \tfrac{1}{5}\right)$;
    \item[(iii)] $|\psi^{(j)}(x)|\leq M$ for all $x\in\mathbb{R}$ and $0\leq j\leq k$.
\end{itemize}

We consider the enumeration of the rational numbers in $[0,1)$ given by
\[
\mathbb{Q} \cap [0,1) = \left\{\tfrac{p_n}{q_n}\right\}_{n \geq 0}
= \left\{0, \tfrac12, \tfrac13, \tfrac23, \tfrac14, \tfrac34, \tfrac15, \dots\right\}.
\]

Let $\varepsilon>0$ be given and define
\[
C:=\left\lceil \frac{(k+1)MA_t}{\varepsilon}\right\rceil,
\quad\text{and}\quad
A_t:=\sum_{n\geq1}\frac{1}{F_n^{t-2k}},
\]
where $F_n$ denotes the $n$th Fibonacci number.

For each $m\in\mathbb{Z}$ and $n\geq 0$, define
\[
\psi_{m,n}(x) := \frac{\lambda_{m,n}}{C\lfloor q_n^{t}\rfloor} 
\, \psi\!\left(q_n^2\left(x-m - \frac{p_n}{q_n}\right)\right),
\]
where $\lambda_{m,n} \in [-1,1]$ will be chosen later.

Since $\operatorname{supp} \psi = \left[-\tfrac{1}{4}, \tfrac{1}{4}\right]$, we have $\psi_{m,n}(x)=0$ whenever
\[
\left|x-m-\frac{p_n}{q_n}\right|>\frac{1}{4q_n^2}.
\]
Hence, if $\psi_{m,n}(x)\neq 0$, then $p_n/q_n$ is a convergent of the continued fraction expansion of $x-m$.

We now define
\[
f(x)=g(x)+\sum_{m\in\mathbb{Z}}\sum_{n\ge 0}\psi_{m,n}(x),
\]
and we will choose $\lambda_{m,n}\in[-1,1]$ so that $f$ satisfies the desired properties.

Fix $m\in\mathbb{Z}$. Note that $\psi_{\ell,n}(m)=0$ for all $\ell\neq m$  and  all $n> 0$, while $\psi_{m,0}(m)=\lambda_{m,0}/C$. Thus, since $C\in \mathbb{N}$,
\[
f(m)=g(m)+\frac{\lambda_{m,0}}{C}.
\]
Since $g(m)\in\mathbb{R}$, there exist at least two choices of $\lambda_{m,0}\in[-1,1]$ such that $f(m)\in\mathbb{Q}$ and
\[
\operatorname{den}(f(m))\le C.
\]

Suppose now that $\lambda_{m,0},\ldots,\lambda_{m,n-1}$ have been chosen so that
\[
f\!\left(m+\frac{p_j}{q_j}\right)\in\mathbb{Q}
\quad \text{and} \quad
\operatorname{den}\!\left(f\!\left(m+\frac{p_j}{q_j}\right)\right)\le C q_j^t
\]
for all $0\le j\le n-1$.

Evaluating at $m+\frac{p_n}{q_n}$, we obtain
\[
f\!\left(m+\frac{p_n}{q_n}\right)
=
g\!\left(m+\frac{p_n}{q_n}\right)
+
\sum_{j=0}^{n-1}\psi_{m,j}\!\left(m+\frac{p_n}{q_n}\right)
+
\frac{\lambda_{m,n}}{C\lfloor q_n^{t}\rfloor}.
\]
The first two terms are already fixed real numbers. By varying $\lambda_{m,n}\in[-1,1]$, we can ensure that
\[
f\!\left(m+\frac{p_n}{q_n}\right)\in\mathbb{Q}
\quad \text{and} \quad
\operatorname{den}\!\left(f\!\left(m+\frac{p_n}{q_n}\right)\right)\le C q_n^t.
\]
Moreover, there are at least two such choices.

Next, we prove that $f\in C^{k}(\mathbb{R})$. By the chain rule, for all $m\in\mathbb{Z}$, $n\ge 0$, and $j\ge 0$, we have
\begin{equation}\label{psij}
    \psi_{m,n}^{(j)}(x) = 
\frac{\lambda_{m,n} q_n^{2j}}{C\lfloor q_n^{t}\rfloor}
\,
\psi^{(j)}\!\left(q_n^2\left(x-m - \frac{p_n}{q_n}\right)\right).
\end{equation}

Let
\[
\tilde f(x) := \sum_{m \in \mathbb{Z}} \sum_{n \geq 0} \psi_{m,n}(x).
\]
Then, by \eqref{psij},
\begin{equation}\label{series}
    \tilde f^{(j)}(x)=\sum_{m \in \mathbb{Z}}\sum_{n \geq 0} \psi^{(j)}_{m,n}(x).
\end{equation}

If 
\[
\psi^{(j)}\left(q_n^2\left(x-m - \frac{p_n}{q_n}\right)\right)\neq 0,
\]
then $p_n/q_n$ is a convergent of $x-m$. Therefore,
\begin{equation}\label{inequality}
     |\tilde f^{(j)}(x)|\leq \sum_{i \geq 0} \frac{ Mq_{n_i}^{2j}}{C\lfloor q_{n_i}^{t}\rfloor},
\end{equation}
where $\{p_{n_i}/q_{n_i}\}$ are the convergents of $x$.

By continued fraction theory, $q_{n_i}\geq F_i\geq \varphi^{i-2}$, where $\varphi=(1+\sqrt{5})/2$. Hence, the series in \eqref{inequality} converges for all $0\leq j < t/2$. Since $t>2k$, it follows that $\tilde f\in C^k(\mathbb{R})$, and therefore $f\in C^k(\mathbb{R})$.

Now, we estimate the distance
\[
|f^{(j)}(x)-g^{(j)}(x)| = |\tilde f^{(j)}(x)|,
\]
and thus
\[
\sum_{j=0}^{k}\sup_{|x|\le m}|f^{(j)}(x)-g^{(j)}(x)|
\leq \sum_{j=0}^{k}\sum_{i \geq 0} \frac{ Mq_{n_i}^{2j}}{C\lfloor q_{n_i}^{t}\rfloor}.
\]
Using $q_{n_i}\geq F_i$, we obtain
\[
\sum_{j=0}^{k}\sum_{i \geq 0} \frac{ Mq_{n_i}^{2j}}{C\lfloor q_{n_i}^{t}\rfloor}
\le \frac{M(k+1)}{C}\sum_{n\geq1}\frac{1}{F_n^{t-2k}}
<\varepsilon.
\]
Since at each step of the construction there are at least two choices for $\lambda_{m,n}$, there are uncountably many such functions. 

Finally, we remark that for all $\alpha \in C(2)$ and for all rational $p/q\in \mathbb{Q}$ we have that
\[
\left|\alpha-\frac{p}{q}\right|>\frac{1}{4q^2},
\]
by the theory of continued fractions. Therefore, $\psi_{m,n}(\alpha)=0$ for all $m,n\in\mathbb{Z}$, with $n\geq 0$. Hence $\tilde f(\alpha)=0$, and $f(\alpha)=g(\alpha)$. This completes the proof.
\qed

\subsection{$C^k$-functions mapping Liouville numbers into themselves}\label{sub3.2}

The next result provides a general criterion ensuring that $C^1$-functions mapping rational numbers into rational numbers preserve Liouville-type properties under mild growth conditions on the denominators.

\begin{lemma}\label{lemma:liouville-preservation}
Let $\Psi:\mathbb{N}\to[0,+\infty)$ be a function such that $\Psi(q)\gg q$ for all sufficiently large $q$, and define
\[
\mathbb{L}_{\Psi}=\left\{\xi\in\mathbb{R}:\; \left|\xi-\frac{p}{q}\right|<\frac{1}{(\Psi(q))^n}
\ \text{for infinitely many } \frac{p}{q}\in\mathbb{Q}\right\}.
\]
Let $\Omega\subseteq\mathbb{R}$ be an open set, and let $f\in C^1(\Omega)$ satisfy
\[
f(\mathbb{Q}\cap\Omega)\subseteq\mathbb{Q}
\quad\text{and}\quad
\operatorname{den}(f(p/q))\le \Psi(q)
\]
for every rational number $p/q\in\mathbb{Q}\cap\Omega$ with $q$ sufficiently large. Then
\[
f(\mathbb{L}_{\Psi}\cap\Omega)\subseteq \mathbb{L}\cup\mathbb{Q}.
\]
Moreover:
\begin{enumerate}[(i)]
    \item If $\operatorname{den}(f(p/q))\to\infty$ as $q\to\infty$, then
    \[
    f(\mathbb{L}_{\Psi}\cap\Omega)\subseteq \mathbb{L}.
    \]
    
    \item If $\Psi(q)\ll q^t$ for some $t>0$, then $\mathbb{L}_{\Psi}=\mathbb{L}$ and hence
    \[
    f(\mathbb{L}\cap\Omega)\subseteq \mathbb{L}.
    \]
\end{enumerate}
\end{lemma}

\begin{proof}
Let $\xi\in \mathbb{L}_{\Psi}\cap\Omega$. Since $\Omega$ is open, there exists a compact interval $I\subseteq \Omega$ containing $\xi$ in its interior. For all rational numbers $p/q$ sufficiently close to $\xi$, we have $p/q\in I$.

By definition of $\mathbb{L}_{\Psi}$, there exists a sequence $(p_n/q_n)_{n\geq1}$ such that
\[
0<\left|\xi-\frac{p_n}{q_n}\right|<\frac{1}{(\Psi(q_n))^n}.
\]

By the Mean Value Theorem, for each $n$ there exists $\theta_n$ between $\xi$ and $p_n/q_n$ such that
\[
\left|f(\xi)-f\!\left(\frac{p_n}{q_n}\right)\right|
= |f'(\theta_n)|\left|\xi-\frac{p_n}{q_n}\right|.
\]
Since $f'\in C^0(\Omega)$, it is bounded on $I$, and therefore there exists $M>0$ such that
\[
\left|f(\xi)-f\left(\frac{p_n}{q_n}\right)\right|
\le M \left|\xi-\frac{p_n}{q_n}\right|
< \frac{M}{(\Psi(q_n))^{n}}.
\]

Write $f(p_n/q_n)=a_n/b_n$ in lowest terms. Then $b_n\le \Psi(q_n)$, and hence
\[
\left|f(\xi)-\frac{a_n}{b_n}\right|
\ll \frac{1}{b_n^{\,n}}.
\]

Therefore, either $f(\xi)=a_n/b_n$ for infinitely many $n$, in which case $f(\xi)\in\mathbb{Q}$, or $f(\xi)$ admits infinitely many rational approximations of Liouville type, which implies that $f(\xi)\in\mathbb{L}$. This proves the first assertion.

If $\operatorname{den}(f(p/q))\to\infty$ as $q\to\infty$, the rational case cannot occur, and hence $f(\xi)\in\mathbb{L}$, proving (i). Finally, if $\Psi(q)\ll q^t$, then $\mathbb{L}_{\Psi}=\mathbb{L}$, and (ii) follows.
\end{proof}

Note that if $\alpha\in\Omega$ and $\varepsilon>0$ are such that
\[
0<\left|\alpha-\frac{p}{q}\right|<\frac{1}{(\Psi(q))^{2+\varepsilon}}
\]
for infinitely many rational numbers $p/q\in\mathbb{Q}\cap\Omega$, arguing as in the proof, one obtains
\[
\left|f(\alpha)-\frac{a}{b}\right|
\ll \frac{1}{b^{\,2+\varepsilon}}.
\]
Hence, by Roth's theorem, either $f(\alpha)\in\mathbb{Q}$ or $f(\alpha)$ is transcendental. In particular, if $\operatorname{den}(f(p/q))\to\infty$ as $q\to\infty$, then $f(\alpha)$ is transcendental.

\subsection{Proof of Corollary \ref{corol1}}\label{sub3.3}

Let $k\geq 1$, let $g\in C^k(\mathbb R)$, and let $\varepsilon>0$. Fix a real number $t>2k$. By Theorem~\ref{theo2}, there exist uncountably many functions $f\in C^k(\mathbb R)$ such that
\[
d(f,g)<\varepsilon,\qquad f(\mathbb Q)\subseteq \mathbb Q,
\]
and
\[
\operatorname{den}\!\left(f\!\left(\frac pq\right)\right)\leq Cq^t
\]
for every rational number $p/q$ in lowest terms. Moreover, all these functions satisfy
\[
f(x)=g(x)\qquad \text{for all }x\in C(2).
\]

Set $\Psi(q)=Cq^t$. Then $\Psi(q)\gg q$ and $\Psi(q)\ll q^t$. Hence, by Lemma~\ref{lemma:liouville-preservation}, we obtain
\[
f(\mathbb L)\subseteq \mathbb L\cup\mathbb Q.
\]

Furthermore, by construction in Theorem~\ref{theo2}, we may choose the functions $f$ so that
\[
\operatorname{den}(f(p/q))\to\infty \quad \text{as } q\to\infty.
\]
Therefore, applying Lemma~\ref{lemma:liouville-preservation}(i), we conclude that
\[
f(\mathbb L)\subseteq \mathbb L.
\]

It remains to prove that uncountably many of these functions are transcendental. Suppose that $f_1$ and $f_2$ are two algebraic functions among the above family. Then each $f_i$ is real-analytic on the complement of a discrete set of singular points. Since the union of two discrete subsets of $\mathbb R$ is again discrete, there exists a nonempty open interval $I\subset \mathbb R$ on which both $f_1$ and $f_2$ are real-analytic. Moreover, since $C(2)$ is perfect and therefore uncountable, we may choose $I$ so that $C(2)\cap I\neq\emptyset$, and thus $C(2)\cap I$ has accumulation points in $I$.

On the other hand,
\[
f_1(x)=f_2(x)\qquad \text{for all }x\in C(2).
\]
Hence the analytic function $f_1-f_2$ vanishes on the set $C(2)\cap I$, which has an accumulation point in $I$. By the identity principle for analytic functions, it follows that
\[
f_1-f_2\equiv 0 \qquad \text{on } I.
\]

Since $f_1-f_2$ is an algebraic function and vanishes on a nonempty open interval, it must vanish identically on its domain. Therefore $f_1=f_2$, showing that at most one function in the above uncountable family can be algebraic.

Consequently, after removing this possible exceptional function, there still remain uncountably many functions $f\in C^k(\mathbb R)$ such that
\[
d(f,g)<\varepsilon,\qquad f(\mathbb L)\subseteq \mathbb L,
\]
and $f$ is transcendental. This proves that $\mathcal L_k$ is dense in $C^k(\mathbb R)$. Since, for each pair $(g,\varepsilon)$, the above construction yields uncountably many such functions, it follows that $\mathcal L_k$ is uncountable. This completes the proof.
\qed

A natural question arising from the results of this paper concerns the existence of functions mapping algebraic numbers of bounded degree and height into algebraic numbers whose height remains polynomially bounded (possibly with controlled degree), as well as determining the maximal differentiability under which such constructions can be achieved. It is also natural to ask whether one can impose analogous arithmetic constraints on the derivatives of such functions. The authors believe that such constructions should be possible.
\section*{Acknowledgements}

The first author gratefully acknowledges the financial support of the Fundação de Estudos em Ciências Matemáticas (FEMAT) under the Incentive for Scientific Production (Edital No.~01/2025).

\bibliographystyle{alpha}
\bibliography{bibliography} 
\end{document}